\documentclass[12pt]{extarticle}
\usepackage{amsmath, amsthm, amssymb, comment, graphicx, svg, url, xcolor} 
\usepackage{hyperref}
\hypersetup{
colorlinks=true,
linkcolor=blue,
filecolor=magenta,
urlcolor=red,
citecolor=gray,
}
\usepackage[capitalise]{cleveref}
\graphicspath{{./Images}}

\newtheorem{theorem}{Theorem}[section]
\newtheorem{corollary}{Corollary}[section]
\newtheorem{proposition}{Proposition}[section]
\theoremstyle{definition}
\newtheorem{definition}{Definition}[section]

\title{Graph Shadows and Edge-Regular Graphs}
\author{
        Jared DeLeo\footnote{Auburn University, jmd0150@auburn.edu}\\
        Department of Mathematics and Statistics\\
        Auburn University\\
        Auburn, AL 36830
        }
\date{April 11, 2025}

\begin{document}

\maketitle

{\sc Abstract:} The definition of edge-regularity in graphs is a relaxation of the definition of strong regularity, so strongly regular graphs are edge-regular and, not surprisingly, the family of edge-regular graphs is much larger and more diverse than that of the strongly regular.

In \cite{deleo}, a few methods of constructing new graphs from old are of use. One of these is the unary ``graph shadow'' operation. Here, this operation is generalized, and then generalized again, and conditions are given under which application of the new operations to edge-regular graphs result in edge-regular graphs. Also, some attention to strongly regular graphs is given.\\


\section{Introduction}

An \textit{edge-regular} graph is a regular graph $G$ such that for some $\lambda \geq 0$, for all $uv \in E(G)$, $|N(u) \cap N(v)| = \lambda$. The set of graphs on $n$ vertices, regular of degree $d$, satisfying the edge-regularity requirement with parameter $\lambda$, will be denoted $ER(n,d,\lambda)$. A \textit{strongly regular} graph is an edge-regular graph $G$ in which for some $\mu \geq 0$, for all $x,y \in V(G), x \not= y$, such that $xy \not\in E(G)$, $|N(x) \cap N(y)| = \mu$. The set of graphs in $ER(n,d,\lambda)$ satisfying the additional strong regularity requirement with parameter $\mu$ will be denoted $SR(n,d,\lambda,\mu)$.

Edge-regular graphs with fixed $\lambda$ have been studied by Glorioso \cite{Glorioso} (when $\lambda = 2$), Bragan \cite{Bragan} (when $\lambda = 1$), and Guest et al. \cite{RCA} (when $\lambda = 1$ and some cases when $\lambda > 1$). All of these publications include constructions for edge-regular graphs with the given $\lambda$ value, with a particular emphasis on RCA graphs, which are edge-regular graphs in which every maximal clique is maximum. Further, edge-regular graphs satisfying $d - \lambda \leq 3$ have been fully characterized by Johnson, Myrvold, and Roblee \cite{ExtremalER}.

Another topic in interest of edge-regular graphs, is using graph products to produce new edge-regular graphs from old. Glorioso \cite{Glorioso} fully characterized the  edge-regular Cartesian, Tensor, Strong, and Lexicographic products of edge-regular graphs. The author expanded upon the Cartesian and Tensor products in \cite{deleo} to characterize preservation by these products of edge-regular graphs with a uniform shared neighborhood structure (USNS).

A different type of graph operation, the \textit{shadow} of a graph, is formally defined and partially studied by the author in \cite{deleo} and by Asmiati et al. in \cite{Shadow}. The goal of this paper is to generalize the definition of the shadow of a graph as a graph operation, and to determine when the generalized operation preserves regularity, edge-regularity, and strong regularity of finite, simple graphs.

\section{$(m,x)$-Shadows}

We define the graph shadow operation below. In this definition, and throughout this paper, when $u \in V(G)$ and $G_1, \dots, G_m$ are copies of $G$, the vertex of $G_i$ playing the role of $u$ will be denoted $u_i$.

For a positive integer $x$ and $v \in V(G)$, $N_G^x(v)$ is the $x$-distance neighborhood of vertex $v$ in $G$. If no graph $G$ is specified, then it will be apparent from the context in what graph the $x$-distance neighborhood is being considered. If no $x$ is specified, then $x$ is understood to be 1. For an example, refer to figure \ref{fig:D23Of4Cycle}.

\begin{definition}\label{def: Shadow}
    Given a finite graph $G$ and $x \geq 1$, $m \geq 2$, the $(m,x)$\textit{-shadow of} $G$, denoted $D_m^x(G)$, is the simple graph whose vertices are in $m$ distinct copies of $G$, say $G_1, G_2, \dots, G_m$; $V(D_m^x(G)) = \bigcup_{i=1}^{m}V(G_i)$ and the edge set is $E(D_m^x(G)) = \{u_iv_j | 1 \leq i < j \leq m, u_i \in V(G_i), v_j \in V(G_j), u \in N_G^x(v)\} \cup \{u_iv_i | 1 \leq i \leq m, u_i,v_i \in V(G_i), uv \in E(G)\}$.
\end{definition}

\begin{figure}[h]
    \centering
    \includegraphics[width=.2\paperwidth]{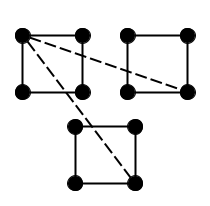}
    \hspace{2cm}
    \includegraphics[width=.2\paperwidth]{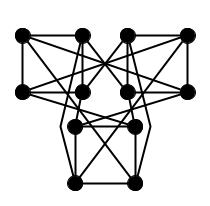}
    \caption{Distance 2 vertices from one vertex of $C_4$ to 2 other copies of $C_4$ (left).
    Edges are added between distance 2 vertices in different copies of $C_4$ to obtain $D_3^2(C_4)$ (right).}
    \label{fig:D23Of4Cycle}
\end{figure}
\begin{theorem}\label{thm:charShadowER}
    Given $G \in ER(n,d,\lambda)$, then for $x > 0$ and $m > 1$, $D_m^x(G)$ is edge- regular if and only if the following conditions hold for some nonnegative integers $d_x, \lambda_x$:\\

    1. For all $v \in V(G)$, $|N^x(v)| = d_x$.
    
    2. For all $u,v \in V(G)$ such that $u \sim v$ in $G$, $|N^x(u) \cap N^x(v)| = \lambda_x$.
    
    3. For all $v,w \in V(G)$ such that $w \in N^x(v)$,\\ $|N^x(v) \cap N(w)| + |N(v) \cap N^x(w)| + (m-2)|N^x(v) \cap N^x(w)| = \lambda + (m-1)\lambda_x$
\end{theorem}
\begin{proof}
    Suppose $G \in ER(n,d,\lambda)$ and $D_m^x(G)$ is edge-regular, where $m \geq 2$ and $x \geq 1$. Define $d_x(v) = |N^x(v)|$. Then for a vertex $v \in V(D_m^x(G))$, $\deg(v) = d + (m-1)d_x(v)$. As $D_m^x(G)$ is edge-regular by assumption, then $\deg(u) = \deg(v)$ for all $u,v \in V(D_m^x(G))$. So, $d + (m-1)d_x(u) = d + (m-1)d_x(v)$ implies that $d_x(u) = d_x(v) = d_x$ for some constant $d_x$, so condition 1 is met.

    Now define $\lambda_x(u,v) = |N^x(u) \cap N^x(v)|$. Then for $u,v \in V(G)$ such that $u$ and $v$ are adjacent, $|N_{D_m^x(G)}(u) \cap N_{D_m^x(G)}(v)| = \lambda + (m-1)\lambda_x(u,v)$. As $D_m^x(G)$ is edge-regular by assumption, then for all $u,v,y,z \in V(D_m^x(G))$ such that $u$ is adjacent to $v$ and $y$ is adjacent to $z$, $|N_{D_m^x(G)}(u) \cap N_{D_m^x(G)}(v)| = |N_{D_m^x(G)}(y) \cap N_{D_m^x(G)}(z)|$. So, $\lambda + (m-1)\lambda_x(u,v) = \lambda + (m-1)\lambda_x(y,z)$ implies that $\lambda_x(u,v) = \lambda_x(y,z) = \lambda_x$ for some constant $\lambda_x$, so condition 2 is met.

    Now consider adjacent vertices in $D_m^x(G)$ in different copies of $G$, say $v$ and $w'$, where $w'$ is a copy of a vertex $w \in N_G^x(v)$. Then $v$ and $w'$ share $|N_G(v) \cap N_G^x(w)|$ vertices in the copy of $G$ containing $v$. Likewise, $|N_G^x(v) \cap N_G(w)|$ vertices are shared in the copy of $G$ containing $w'$. In each of the remaining $m-2$ copies of $G$, $v$ and $w'$ share $|N_G^x(v) \cap N_G^x(w)|$ vertices. As $D_m^x(G)$ is edge-regular by assumption, then $|N_{D_m^x(G)}(v) \cap N_{D_m^x(G)}(w')|$ is equal to the number of vertices shared in $D_m^x(G)$ by two adjacent vertices in the same copy of $G$. So, $|N_{D_m^x(G)}(v) \cap N_{D_m^x(G)}(w')| = |N_G(v) \cap N_G^x(w)| + |N_G^x(v) \cap N_G(w)| + (m-2)|N_{D_m^x(G)}(v) \cap N_{D_m^x(G)}(w')| = \lambda + (m-1)\lambda_x$ by condition 2. Thus, all conditions are met.

    The proof of the converse is straightforward, using the same arguments as in the forward direction.
\end{proof}

\cref{thm:charShadowER} generalizes one implication of a result in \cite{deleo}, which asserts that when $x = 1$, $D_m^1(G)$ is edge-regular if $G$ is edge-regular. When $x = 1$ and $G \in ER(n,d,\lambda)$, clearly conditions 1, 2, and 3 hold for any $m$ with $d_1 = d$ and $\lambda_1 = \lambda$.

But for $m,x > 1$, the edge-regularity of $G$ does not imply the edge-regularity of $D_m^x(G)$. An example is given in figure \ref{fig:D22Of5Cycle}. In this example, condition 3 of \cref{thm:charShadowER} does not hold.

\begin{figure}[h]
    \centering
    \includegraphics[width = .08\paperwidth]{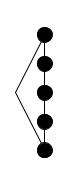}
    \hspace{2cm}
    \includegraphics[width = .16\paperwidth]{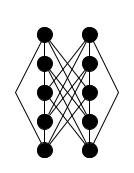}
    \caption{$C_5 \in ER(5,2,0)$ (left) and $D_2^2(C_5)$ which is not edge-regular (right).}
    \label{fig:D22Of5Cycle}
\end{figure}

In the case $x = 0$, if we agree that for each $v \in V(G)$, $N^0(v) = \{v\}$, and read the definition of adjacency in $D_m^0(G)$ as it was for $D_m^x(G)$, $x > 0$, then we see that for each $v \in V(G)$, each of its clones in any of the $m$ copies of $G$ constituting $D_m^0(G)$ is adjacent to each of the $m-1$ other clones of $v$ in the other copies of $G$, and to no other vertices in those other copies. We enshrine this description in the following proposition.

\begin{proposition}\label{prop:Cartesian}
    For each graph $G$ and integer $m \geq 1$, $D_m^0(G) \cong G \square K_m$.
\end{proposition}
\begin{corollary}\label{cor:0Shadow}
    If $G \in ER(n,d,\lambda)$ then $H = D_m^0(G)$ is of order $mn$ and regular of degree $m-1+d$. If $m > 1$ then $H$ is edge-regular if and only if $m-2 = \lambda$, in which case $H \in ER(mn,m-1+d,m-2) = ER(mn,m-1+d,\lambda)$.
\end{corollary}

Surprisingly, \cref{thm:charShadowER} holds when $x = 0$. Clearly, for any $G$, conditions 1 and 2 are met with $d_0 = 1$ and $\lambda_0 = 0$. For condition 3, observe that $w \in N^0(v)$ implies that $w = v$. Then the equation in condition 3 collapses to $0 + 0 + (m-2)(1) = \lambda + 0$, or $m - 2 = \lambda$, which is precisely the necessary and sufficient condition for $D_m^0(G)$ to be edge-regular when $G$ is edge-regular (\cref{cor:0Shadow}).

We now show how a class of edge-regular graphs, under the shadow graph operation, is used to build larger edge-regular graphs. Let $C_n$ denote the cycle graph on $n$ vertices, $d(u,v)$ denote the distance between two vertices $u$ and $v$, and let $r(G)$ denote the \textit{radius} of a graph $G$. $r(G)$ is the minimum \textit{eccentricity} among the vertices of $G$. The eccentricity $\epsilon(v)$ of $v \in V(G)$ is the maximum of the distances in $G$ from $v$ to vertices of $G$ (these definitions require $G$ to be connected). It is well known that, for $n \geq 3$, $r(C_n) = \lfloor \frac{n}{2} \rfloor$.

\begin{corollary}
    $D_m^x(C_n)$ is edge-regular for all integers $n \geq 3$, $m \geq 2$, $1 \leq x \leq r(C_n)$ except in either of the following cases for some integer $k \geq 2$:
    \begin{enumerate}
        \item $n = 3k, x = k, m \not= 2$.
        \vspace{-.3cm}
        \item $n = 2k+1, x = k, m \not= 3$.
    \end{enumerate}
\end{corollary}
\begin{proof}
    Consider $n \equiv 0 \pmod 2$, so $C_n$ is an even cycle with radius $\frac{n}{2}$. Then for $1 \leq x \leq \frac{n}{2}$, $v \in V(C_n)$, $|N^x(v)| = 1$ if $x = \frac{n}{2}$, or $|N^x(v)| = 2$ if $x < \frac{n}{2}$. Further, for $uv \in E(C_n)$, $|N^x(u) \cap N^x(v)| = 0$. Additionally, for $w \in N^x(v)$, notice that $|N^x(v) \cap N(w)| = |N(v) \cap N^x(w)| = 0$; $|N^x(v) \cap N^x(w)| = 1$ if $x = \frac{n}{3}$, and $|N^x(v) \cap N^x(w)| = 0$ if $x \not= \frac{n}{3}$.
    
    Then by \cref{thm:charShadowER}, for $n \equiv 0 \pmod 2$, $D_m^x(C_n)$ is edge-regular if $x \not= \frac{n}{3}$. When $x = \frac{n}{3}$, $D_m^x(C_n)$ is edge-regular only when $m = 2$.

    Now consider $n \equiv 1 \pmod 2$, so $C_n$ is an odd cycle with radius $\lfloor\frac{n}{2}\rfloor$. Then for $1 \leq x \leq \lfloor\frac{n}{2}\rfloor$, $v \in V(C_n)$, $|N^x(v)| = 2$. Further, for $uv \in E(C_n)$, $|N^x(u) \cap N^x(v)| = 1$ if $x = \lfloor\frac{n}{2}\rfloor$, or $|N^x(u) \cap N^x(v)| = 0$ if $x < \lfloor\frac{n}{2}\rfloor$. Additionally, for $w \in N^x(v)$, notice that $|N^x(v) \cap N(w)| = |N(v) \cap N^x(w)| = 1$ if $x = \lfloor\frac{n}{2}\rfloor$, or $|N^x(v) \cap N(w)| = |N(v) \cap N^x(w)| = 0$ if $x \not= \lfloor\frac{n}{2}\rfloor$; $|N^x(v) \cap N^x(w)| = 1$ if $x = \frac{n}{3}$, and $|N^x(v) \cap N^x(w)| = 0$ if $x \not= \frac{n}{3}$.
    
    Then by \cref{thm:charShadowER}, for $n \equiv 1 \pmod 2$, $D_m^x(C_n)$ is edge-regular if $x \not\in \{\frac{n}{3}, \lfloor\frac{n}{2}\rfloor\}$. If $x = \frac{n}{3}$, $D_m^x(C_n)$ is edge-regular only when $m = 2$. If $x = \lfloor\frac{n}{2}\rfloor$, $D_m^x(C_n)$ is edge-regular only when $m = 3$.
\end{proof}

\begin{figure}[h]
    \centering
    \includegraphics[width = .25\paperwidth]{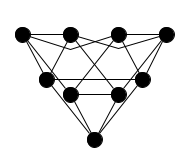}
    \caption{An example of a 0-distance shadow, $D_3^0(K_3) \in ER(9,4,1)$.}
    \label{fig:D03OfK3}
\end{figure}

\begin{theorem}\label{thm:charShadow}
Suppose $m,x \geq 1$ are integers, $G$ is a simple, connected graph of order $n$ and $D_m^x(G)$ is the $(m,x)$-shadow of $G$. Consider the following conditions for some nonnegative integers $d'_x$, $\lambda'_x$, and $\mu'_x$:\\

1. For all $v \in V(G)$, $|N(v)| + (m-1)|N^x(v)| = d'_x$\\

2. For all $u,v \in V(G)$ such that $u \sim v$ in $G$,
\begin{align*}
    |N(u) \cap N(v)| + (m-1)|N^x(u) \cap N^x(v)| = \lambda'_x
\end{align*}

3. For all $v,w \in V(G)$ such that $w \in N_G^x(v)$,
\begin{align*}
    |N(v) \cap N^x(w)| + |N^x(v) \cap N(w)| + (m-2)|N^x(v) \cap N^x(w)| = \lambda'_x
\end{align*}

4. For all $u,v \in V(G)$ such that $u \not\sim v$ in $G$,
\begin{align*}
    |N(u) \cap N(v)| + (m-1)|N^x(u) \cap N^x(v)| = \mu'_x
\end{align*}

5. For all $v,w \in V(G), w \not\in N^x(v)$,
\begin{align*}
    |N(v) \cap N^x(w)| + |N^x(v) \cap N(w)|
    + (m-2)|N^x(v) \cap N^x(w)| = \mu'_x
\end{align*}

Condition 1 is met if and only if $D_m^x(G)$ regular of degree $d'_x$.

Conditions 1, 2, and 3 are met if and only if $D_m^x(G) \in ER(mn,d'_x,\lambda'_x)$.

All conditions are met if and only if $D_m^x(G) \in SR(mn,d'_x,\lambda'_x,\mu'_x)$.
\end{theorem}
\begin{proof}
    Suppose condition 1 holds. For all $v_i \in V(D_m^x(G))$, $v_i$ is adjacent to $N_G(v_i)$ in one copy of $G$ and adjacent to $N_G^x(v_i)$ in $m-1$ copies of $G$. Then $|N_{D_m^x(G)}(v_i)| = |N(v_i)| + (m-1)|N^x(v_i)| = d'_x$. So $D_m^x(G)$ is regular of degree $d'_x$.

    Now suppose conditions 1, 2, and 3 hold. As stated previously, if condition 1 holds, then $D_m^x(G)$ is regular. There are two ``types" of adjacencies in $D_m^x(G)$. One type are adjacencies between vertices in the same copy of $G$. Consider two vertices of this type, say $u$ and $v$. Then as $u$ and $v$ share $\lambda$ neighbors in the same copy of $G$ and $x$-distance neighbors in $m-1$ copies of $G$, $|N_{D_m^x(G)}(u) \cap N_{D_m^x(G)}(v)| = \lambda + (m-1)|N^x(u) \cap N^x(v)| = \lambda'_x$ by condition 2.

    The second type of adjacencies are between vertices in distinct copies of $G$. Consider two vertices of this type, say $v$ and $w'$, where $w'$ is a copy of $w$, which is in the same copy of $G$ as $v$. That is, $w \in N^x(v)$. Then in the copy of $G$ containing $v$, the number of common neighbors of $v$ and $w'$ is the number of neighbors of $v$ that are $x$-distance neighbors of $w$. In the copy of $G$ containing $w'$, the number of common neighbors is neighbors of $w$ that are $x$-distance neighbors of $v$. In the remaining $m-2$ copies of $G$ not containing $v$ or $w'$, the common neighbors are $x$-distance neighbors of both $v$ and $w$. Then $|N_{D_m^x(G)}(v) \cap N_{D_m^x(G)}(w')| = |N^x(v) \cap N(w)| + |N(v) \cap N^x(w)| + (m-2)|N^x(v) \cap N^x(w)| = \lambda'_x$ by condition 3.

    Thus, as all pairs of adjacent vertices have $\lambda'_x$ common neighbors, $D_m^x(G)$ is edge-regular.

    Now suppose all conditions are met. As stated previously, since conditions 1, 2, and 3 are met, $D_m^x(G) \in ER(mn,d'_x,\lambda'_x)$.

    Consider non-adjacent vertices in the same copy of $G$, say $u$ and $v$. In the same copy of $G$, the number of common neighbors is $|N(u) \cap N(v)|$. In the other $m-1$ copies of $G$, the number of common neighbors is $|N^x(u) \cap N^x(v)|$. So $|N_{D_m^x(G)}(u) \cap N_{D_m^x(G)}(v)| = |N(u) \cap N(v)| + (m-1)|N^x(u) \cap N^x(v)| = \mu'_x$ by condition 4.

    Consider non-adjacent vertices in distinct copies of $G$, say $v$ and $w'$, where $w'$ is a copy of $w \in V(G)$. That is, $w \not\in N_G^x(v)$. In the copy of $G$ containing $v$, the number of common neighbors of $v$ and $w'$ is the number of neighbors of $v$ in the $x$-distance neighborhood of $w$. In the copy of $G$ containing $w'$, the number of common neighbors is the number of neighbors of $w$ in the $x$-distance neighborhood of $v$. In the remaining $m-2$ copies of $G$, the number of common neighbors is the number of $x$-distance neighbors of $v$ in the $x$-distance neighborhood of $w$. So $|N_{D_m^x(G)}(v) \cap N_{D_m^x(G)}(w')| = |N^x(v) \cap N(w)| + |N(v) \cap N^x(w)| + (m-2)|N^x(v) \cap N^x(w)| = \mu'_x$ by condition 5.

    As all non-adjacent vertices in $D_m^x(G)$ share $\mu'_x$ common neighbors, $G \in SR(mn,d'_x,\lambda'_x,\mu'_x)$.

    The converse is straightforward.
\end{proof}

\begin{figure}[h]
    \centering
    \includegraphics[width = .18\paperwidth]{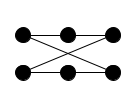}
    \caption{$D_2^2(P_3) \cong C_6 \in ER(6,2,0)$}
    \label{fig:D22OfP3}
\end{figure}

Theorem \ref{thm:charShadow} removes the restriction that $G$ is edge-regular, and allows for any simple, connected graph.  Consider the $(2,2)$-shadow of the path graph on 3 vertices, which is non-regular, shown in figure \ref{fig:D22OfP3}. $D_2^2(P_3) \cong C_6 \in ER(6,2,0)$, which satisfies conditions 1, 2, and 3 in the above theorem.

While the above edge-regular example is an immediate result of \cref{thm:charShadow}, it is unknown for what graphs $G$ and parameters $m,x$ yield strongly regular graphs from $D_m^x(G)$.

\section{$(m,X)$-Shadows}

To further generalize, we consider multiple distances to be used in the definition of the shadow of the graph as opposed to a single distance. This new multi-distance shadow generalizes Definition \ref{def: Shadow}, formally defined below.

\begin{definition}\label{def: MultiShadow}
    Given a finite graph $G$, $X = \{x_1,\dots,x_p\}$, $1 \leq x_i < \dots < x_p$, and $m \geq 2$, the $(m,X)$\textit{-shadow of} $G$, denoted $D_m^X(G)$, is the simple graph whose vertices are in $m$ distinct copies of $G$, say $G_1, G_2, \dots, G_m$; $V(D_m^X(G)) = \bigcup_{i=1}^{m}V(G_i)$ and $E(D_m^X(G)) = \{u_iv_j | i \not= j, u_i \in V(G_i), v_j \in V(G_j), u \in \bigcup_{k = 1}^{p}N_G^{x_k}(v)\} \cup \{u_iv_i | u_i,v_i \in V(G_i), uv \in E(G)\}$.
\end{definition}

As with the case of $D_m^x(G)$ in \cref{thm:charShadowER}, an immediate question arises as to when edge-regularity is preserved in $D_m^X(G)$. The following theorem characterizes when $D_m^X(G)$ is edge-regular, given that $G$ is edge-regular.

\begin{theorem}\label{thm:charMultShadowER}
    Given $G \in ER(n,d,\lambda)$ and $X = \{x_1,x_2,\dots,x_p\}$, then $D_m^X(G)$ is edge-regular if and only if the following conditions are met for some integers $\bar{d}$ and $\bar{\lambda}$:\\

    1. For all $v \in V(G)$, $\displaystyle\sum_{i = 1}^{p}|N^{x_i}(v)| = \Bar{d}$
    
    2. For all $u,v \in V(G)$ such that $u \sim v$ in $G$, $\displaystyle\sum_{i = 1}^{p}\displaystyle\sum_{j = 1}^{p}|N^{x_i}(u) \cap N^{x_j}(v)| = \Bar{\lambda}$
    
    3. For all $v,w \in V(G)$ such that $w \in \displaystyle \bigcup_{q = 1}^{p}N^{x_q}(v)$,\\ $\displaystyle \sum_{i = 1}^{p}|N(v) \cap N^{x_i}(w)| + \sum_{j = 1}^{p}|N^{x_j}(v) \cap N(w)| + (m-2)\sum_{k=1}^{p}\sum_{l = 1}^{p}|N^{x_k}(v) \cap N^{x_l}(w)| = \lambda + (m-1)\Bar{\lambda}$.
\end{theorem}

\begin{proof}
    Suppose $G$ and $D_m^X(G)$ are edge-regular, where $m \geq 1$ and $X = \{x_1,\dots,x_p\}$. Then in $D_m^X(G)$,
    \begin{align*}
        \deg{(v)} &= d + (m-1)|N^{x_1}_G(v)| + (m-1)|N^{x_2}_G(v)| + \cdots + (m-1)|N^{x_p}_G(v)|\\
        &= d + (m-1)\sum_{i = 1}^{p}|N^{x_i}(v)|
    \end{align*}
    As $D_m^X(G)$ is regular by assumption, $\deg(u) = \deg(v)$ for all $u,v \in V(D_m^X(G)$. So,
    \begin{align*}
        d + (m-1)\sum_{i = 1}^{p}|N^{x_i}(u)| &= d + (m-1)\sum_{i = 1}^{p}|N^{x_i}(v)|\\
        \sum_{i = 1}^{p}|N^{x_i}(u)| &= \sum_{i = 1}^{p}|N^{x_i}(v)|
    \end{align*}
    Thus, $\displaystyle\sum_{i = 1}^{p}|N^{x_i}(v)| = \bar{d}$ for all $v \in V(D_m^X(G))$, so condition 1 is met.

    Now define $\lambda_{i,j}(u,v) = |N^{x_i}_G(u) \cap N^{x_j}_G(v)|$. Then for $uv \in E(G)$,
    \begin{align*}
        |N_{D_m^X(G)}(u) \cap N_{D_m^X(G)}(v)| &= \lambda + (m-1)\lambda_{1,1}(u,v) + \cdots + (m-1)\lambda_{1,p}(u,v)\\
        &\hspace{.98cm}+ \cdots\\
        &\hspace{.98cm}+ (m-1)\lambda_{p,1}(u,v) + \cdots + (m-1)\lambda_{p,p}(u,v)\\
        &= \lambda + (m-1)\sum_{i = 1}^{p}\sum_{j = 1}^{p}\lambda_{i,j}(u,v)
    \end{align*}    
    As $D_m^X(G)$ is edge-regular by assumption,
    \begin{align*}
        |N_{D_m^X(G)}(u) \cap N_{D_m^X(G)}(v)| = |N_{D_m^X(G)}(y) \cap N_{D_m^X(G)}(z)|
    \end{align*} for all $uv, yz \in E(G)$. So,
    \begin{align*}
        \lambda + (m-1)\sum_{i = 1}^{p}\sum_{j = 1}^{p}\lambda_{i,j}(u,v) &= \lambda + (m-1)\sum_{i = 1}^{p}\sum_{j = 1}^{p}\lambda_{i,j}(y,z)\\
        \sum_{i = 1}^{p}\sum_{j = 1}^{p}\lambda_{i,j}(u,v) &= \sum_{i = 1}^{p}\sum_{j = 1}^{p}\lambda_{i,j}(y,z)
    \end{align*}
    Thus, for all $uv \in E(G)$, $\displaystyle\sum_{i = 1}^{p}\sum_{j = 1}^{p}\lambda_{i,j}(u,v) = \bar{\lambda}$, so condition 2 is met.

    Now consider adjacent vertices of $D_m^X(G)$ in different copies of $G$, say $v$ and $w'$, where $w'$ is a copy of a vertex $w \in \bigcup_{q = 1}^{p}N_{G}^{x_q}(v)$. Then $v$ and $w'$ share $\sum_{i = 1}^{p}|N_G(v) \cap N_G^{x_i}(w)|$ vertices in the copy of $G$ containing $v$. Likewise, $v$ and $w'$ share $\sum_{j=1}^{p}|N_G^{x_j}(v) \cap N_G(w)|$ vertices in the copy of $G$ containing $w'$. In each of the remaining $m - 2$ copies of $G$, $v$ and $w'$ share $\sum_{k=1}^{p}\sum_{l = 1}^{p}|N_G^{x_k}(v) \cap N_G^{x_l}(w)|$ vertices. As $D_m^X(G)$ is edge-regular by assumption, $|N_{D_m^X(G)}(v) \cap N_{D_m^X(G)}(w')|$ is equal to the number of vertices shared by two adjacent vertices in the same copy of $G$. Thus,
    \begin{align*}
        |N_{D_m^X(G)}(v) \cap N_{D_m^X(G)}(w')| &= \sum_{i = 1}^{p}|N_G(v) \cap N_G^{x_i}(w)| + \sum_{j=1}^{p}|N_G^{x_j}(v) \cap N_G(w)|\\
        &+ (m-2)\sum_{k=1}^{p}\sum_{l = 1}^{p}|N_G^{x_k}(v) \cap N_G^{x_l}(w)|\\
        &= \lambda + (m-1)\bar{\lambda} \text{\hspace{.5cm}(by condition 2)}
    \end{align*}
    Thus, condition 3 is met.

    The converse is straightforward.
\end{proof}

\begin{figure}[h]
    \centering
    \includegraphics[width=.4\paperwidth]{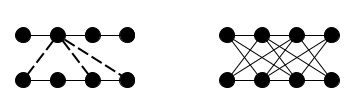}
    \caption{Distance 1 and 2 vertices from one vertex of $P_4$ to another copy of $P_4$ (left).
    Edges added between distance 1 and 2 vertices in different copies of $P_4$ to obtain $D_2^{1,2}(P_4)$ (right).}
    \label{fig:D122OfP4}
\end{figure}

\section{Generalized Graph Shadows}

\begin{definition}\label{def: GeneralizedShadow}
    Given finite graphs $G,H$ and $x \geq 1$, $m = |V(H)|$; let the vertices of $H$ be ordered $w_1,\dots,w_m$. The $(H,x)$\textit{-shadow of} $G$, denoted $D_m^x(G,H)$, is the simple graph whose vertices are in $m$ distinct copies of $G$, say $G_1, G_2, \dots, G_m$. $V(D_m^x(G)) = \bigcup_{i=1}^{m}V(G_i)$ and edge set $E(D_m^x(G)) = \{u_iv_j | i \not= j, u_i \in V(G_i), v_j \in V(G_j), w_i,w_j \in V(H), w_iw_j \in E(H), u \in N_G^x(v)\} \cup \{u_iv_i | u_i,v_i \in V(G_i), uv \in E(G)\}$.
\end{definition}

The definition of $D_m^x(G,H)$ refers to an ordering of $V(H)$, but the isomorphism class of $D_m^x(G,H)$ does not depend on the ordering. Permuting the vertex set of $H$ results in a rearrangement of the list $G_1,\dots,G_m$; while the new $D_m^x(G,H)$ obtained thereby is not identical to the old, the two are obviously isomorphic, as the ends of the edges between $G_i$ and $G_j$, $i \not= j$, in the old are dragged along to the new positions of these copies of $G$.

Note that by this definition, if $|V(G)| = n$, $D_m^x(G) = D_m^x(G,K_n)$. For an example, refer to figure \ref{fig:D4OfP3C4}. The following theorem gives a characterization for when $D_m^x(G,G)$ is edge-regular.

\begin{figure}[h]
    \centering
    \includegraphics[width=.30\paperwidth]{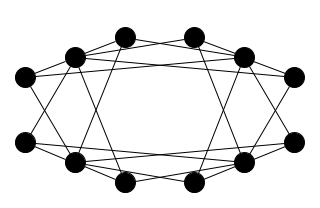}
    \caption{$D_4^1(P_3,C_4)$}
    \label{fig:D4OfP3C4}
\end{figure}

\begin{theorem}\label{thm: GeneralizedGGShadow}
    Given $G \in ER(n,d,\lambda)$, then $D_m^x(G,G)$ is edge-regular if and only if the following conditions are met for some integers $\bar{d}$ and $\bar{\lambda}$:\\

    1. For all $v \in V(G)$, $d(1 + |N^x(v)|) = \Bar{d}$.

    2. For all $u,v \in V(G)$ such that $u \sim v$ in $G$, $\lambda + d|N^x(u) \cap N^x(v)| = \Bar{\lambda}$.

    3. For all $v,w \in V(G)$ such that $w \in N_G^x(v),$
    \begin{align*}
        |N^x(v) \cap N(w)| + |N(v) \cap N^x(w)| + \lambda|N^x(v) \cap N^x(w)| = \Bar{\lambda}.
    \end{align*}
\end{theorem}
\begin{proof}
    To prove the forward direction, suppose $G \in ER(n,d,\lambda)$. Then $v \in V(G)$ is adjacent to $d$ vertices in its own copy of $G$ and is adjacent to $|N^x(v)|$ vertices in $d$ other copies of $G$. Then for $v_i \in V(G)$, $|N_{D_m^x(G,G)}(v_i)| = d + d|N^x(v_i)| = d(1 + |N^x(v_i)|) = \Bar{d_i}$. As $G$ is edge-regular, $\Bar{d_i} = \Bar{d}$ for all $i \in [n]$. Thus, condition 1 is met.

    Consider adjacent vertices $u,v$ in the same copy of $G$. Then $u$ and $v$ have $\lambda$ common neighbors in the copy of $G$ containing them, and have $|N^x(u) \cap N^x(v)|$ common neighbors in $d$ copies of $G$. So $|N_{D_m^x(G,G)}(u) \cap N_{D_m^x(G,G)}(v)| = \lambda + d|N^x(u) \cap N^x(v)| = \Bar{\lambda}_{u,v}$. As $G$ is edge-regular, $\Bar{\lambda}_{u,v} = \Bar{\lambda}$ for all pairs $u,v \in V(G)$ such that $u \sim v$. Thus, condition 2 is met.

    Consider adjacent vertices in distinct copies of $G$, say $v$ and $w'$, where $w'$ is a copy of $w$, which is in the same copy of $G$ as $v$. That is, $w \in N^x(v)$. Then in the copy of $G$ containing $v$, the number of common neighbors of $v$ and $w'$ is $|N^x(v) \cap N(w)|$. In the copy of $G$ containing $w'$, the number of common neighbors is $|N(v) \cap N^x(w)|$. These distinct copies of $G$ are mutually adjacent to $\lambda$ other copies of $G$. In these $\lambda$ copies of $G$, $v$ and $w'$ have $|N^x(v) \cap N^x(w)|$ common neighbors. So $|N_{D_m^x(G,G)}(v) \cap N_{D_m^x(G,G)}(w')| = |N^x(v) \cap N(w)| + |N(v) \cap N^x(w)| + \lambda|N^x(v) \cap N^x(w)| = \Bar{\lambda}_{v,w'}$. As $G$ is edge-regular, then $\Bar{\lambda}_{v,w'} = \Bar{\lambda}$ for all pairs of adjacent vertices $v,w'$ in distinct copies of $G$. Thus, condition 3 is met.

    The converse is straightforward.
\end{proof}

\begin{corollary}\label{cor:GeneralizedShadowERGG}
    If $G \in ER(n,d,\lambda)$, then $D_m(G,G)$ is edge-regular.
\end{corollary}
\begin{proof}
    Given $G \in ER(n,d,\lambda)$, let $x = 1$. Then for a vertex $v \in V(G)$, $d(1 + |N(v)|) = d + (d)^2 = \bar{d}$ for some integer $\bar{d}$. So condition 1 of Theorem \ref{thm: GeneralizedGGShadow} is met. For adjacent vertices $u,v \in V(G)$, $\lambda + d|N(u) \cap N(v)| = \lambda + d\lambda = \bar{\lambda}$ for some integer $\bar{\lambda}$. So condition 2 of Theorem \ref{thm: GeneralizedGGShadow} is met. Finally, condition 3 of Theorem \ref{thm: GeneralizedGGShadow} is trivially met when $x = 1$. Thus, $D_m(G,G)$ is edge-regular.
\end{proof}

\cref{cor:GeneralizedShadowERGG} justifies a way to construct edge-regular graphs which resembles a recursion-like relation in the generalized graph shadow. It remains to be seen, for any edge-regular graph $H$, if $D_m^x(G,H)$ is edge-regular. Characterizing when $D_m^x(G,H)$ is edge-regular for any simple, connected graphs $G$ and $H$ would generalize a number of results in this paper and would provide a framework for construction of regular and strongly-regular graphs. Under the assumption that this characterization exists, extending it to strongly regular graphs would also be of interest.

\newpage

\bibliography{references}
\bibliographystyle{IEEEtran}

\end{document}